\documentclass{amsart}
\usepackage{amssymb,latexsym,pstricks}
\topskip  \numberwithin{equation}{section}
\newtheorem{theorem}{Theorem}[section]

\newtheorem{corollary}[theorem]{Corollary}
\newtheorem{definition}[theorem]{Definition}

\newtheorem{lemma}[theorem]{Lemma}

\begin{document}

\pagenumbering{arabic}
\pagestyle{headings}
\def\sof{\hfill\rule{2mm}{2mm}}
\def\llim{\lim_{n\rightarrow\infty}}
\def\da{\text{-}}

\title{A monotonicity property for generalized Fibonacci sequences}
\maketitle

\begin{center}

\author{Toufik Mansour\\
\small Department of Mathematics, University of Haifa, 31905 Haifa, Israel\\[-0.8ex]
\small\texttt{tmansour@univ.haifa.ac.il}\\[1.8ex]
Mark Shattuck\\
\small Department of Mathematics, University of Tennessee, Knoxville, TN 37996\\[-0.8ex]
\small\texttt{shattuck@math.utk.edu}\\[1.8ex]}
\end{center}

\begin{abstract}
Given $k \geq 2$, let $a_n$ be the sequence defined by the recurrence $a_n=\alpha_1a_{n-1}+\cdots+\alpha_ka_{n-k}$ for $n \geq k$, with initial values $a_0=a_1=\cdots=a_{k-2}=0$ and $a_{k-1}=1$.  We show under a couple of assumptions concerning the constants $\alpha_i$ that the ratio $\frac{\sqrt[n]{a_n}}{\sqrt[n-1]{a_{n-1}}}$ is strictly decreasing for all $n \geq N$, for some $N$ depending on the sequence, and has limit $1$. In particular, this holds in the cases when all of the $\alpha_i$ are unity or when all of the $\alpha_i$ are zero except for the first and last, which are unity.  Furthermore, when $k=3$ or $k=4$, it is shown that one may take $N$ to be an integer less than $12$ in each of these cases.
\end{abstract}



%===========================================================================

%\thispagestyle{empty}
%===========================================================================

\section{Introduction}

In 1982, Firoozbakht conjectured that the sequence $\{\sqrt[n]{p_n}\}_{n\geq1}$ is strictly decreasing, where $p_n$ denotes the $n$-th prime.  A stronger conjecture was later made by Sun \cite{Sun} that in fact
$$\frac{\sqrt[n+1]{p_{n+1}}}{\sqrt[n]{p_n}}<1-\frac{\log\log n}{2n^2}, \qquad n>4,$$
which has been verified for all $n \leq 3.5 \cdot 10^6$.  Inspired by this and \cite{Sun0}, Sun posed several conjectures in \cite{Sun} concerning the monotonicity of sequences of the form $\{\sqrt[n]{y_n}\}_{n \geq N}$, where $\{y_n\}_{n\geq0}$ is a familiar number theoretic or combinatorial sequence.  Partial progress has been made in this direction, including Chen et al. \cite{CGW} for Bernoulli numbers, Hou et al. \cite{HSW} for Fibonacci and derangement numbers, and Wang and Zhu \cite{WZ} for Motzkin and (large) Schr\"{o}der numbers.

Recall that a sequence $\{y_n\}_{n\geq0}$ is said to be (\emph{strictly}) \emph{log concave} (see, e.g., \cite{Br, St}) if the sequence of ratios $\{\frac{y_n}{y_{n-1}}\}_{n\geq 1}$ is (strictly) decreasing.  If the sequence of ratios is increasing, then $y_n$ is said to be \emph{log convex} (see \cite{LW}).  Suppose $A>0$ and $B\neq0$ are integers such that $A^2-4B>0$.  Let $u_n$ denote the sequence defined by the second order recurrence $u_n=Au_{n-1}-Bu_{n-2}$ if $n \geq 2$, with initial values $u_0=0$ and $u_1=1$.  In \cite[Theorem 1.1]{HSW}, it was shown that $\sqrt[n]{u_n}$ is strictly log-concave for all $n \geq N$, for some $N$ depending on the sequence, and has limit $1$.  In the special case $A=1$ and $B=-1$, which corresponds to the Fibonacci sequence, it is shown that one may take $N=5$.  Here, we consider the question of monotonicity of $\frac{\sqrt[n]{a_n}}{\sqrt[n-1]{a_{n-1}}}$ for a class of sequences $a_n$ defined by a more general linear recurrence.

Given $k \geq 2$, let $a_n$ be a sequence of non-negative real numbers defined by the recurrence
\begin{equation}\label{eqa0}
a_n=\alpha_1a_{n-1}+\alpha_2a_{n-2}+\cdots+\alpha_ka_{n-k}, \qquad n \geq k,
\end{equation}
with $a_0=a_1=\cdots=a_{k-2}=0$ and $a_{k-1}=1$.  One combinatorial interpretation for $a_n$, which follows from \cite[Section 3.1]{BQ}, is that it counts the weighted linear tilings of length $n-k+1$ in which the tiles have length at most $k$, where a tile of length $i$ is assigned the weight $\alpha_i$.  It will be shown that the sequence $\{\sqrt[n]{a_n}\}$ is strictly log-concave for all $n$ sufficiently large under a couple of assumptions concerning the constants $\alpha_i$ (see Theorem \ref{th1} below).  As a special case, one obtains the log-concavity result mentioned in the previous paragraph for the second-order sequence $u_n$.

We now recall two well-known classes of recurrences.  Letting $\alpha_1=\alpha_2=\cdots=\alpha_k=1$ in \eqref{eqa0}, one gets the $k$-\emph{Fibonacci} sequence, which we will denote here by $f_n^{(k)}$.  The sequence $f_n^{(k)}$ was first considered by Knuth \cite{Kn} and has been given interpretations in terms of linear tilings \cite[Chapter 3]{BQ} and $k$-filtering linear partitions \cite{Mu}.  When $\alpha_1=\alpha_k=1$ and all other $\alpha_i$ are zero,  one gets a class of sequences known as the $k$-\emph{bonacci} numbers (see, e.g., \cite[Section 3.4]{BQ}), which we will denote by $g_n^{(k)}$.  Note that both $f_n^{(k)}$ and $g_n^{(k)}$ reduce to the usual Fibonacci numbers when $k=2$.  It will be shown that the ratio $\frac{\sqrt[n]{a_n}}{\sqrt[n-1]{a_{n-1}}}$ is decreasing for all $n\geq N$ for some $N$ depending on $k$ whenever $a_n=f_n^{(k)}$ or $g_n^{(k)}$.

In the third section, we consider the special cases of $f_n^{(k)}$ and $g_n^{(k)}$ when $k=3$ and $k=4$ and show that one may take $N$ to be an integer less than $12$ in each of these cases.  Our method will apply to finding the best possible $N$ for any \emph{given} sequence $a_n$ satisfying a recurrence of the form \eqref{eqa0} for which  $\sqrt[n]{a_n}$ is eventually log-concave.

\section{Main results}

Given $k \geq 2$, let $a_n$ be a sequence of non-negative real numbers defined by the recurrence
\begin{equation}\label{eq0}
a_n=\alpha_1a_{n-1}+\alpha_2a_{n-2}+\cdots+\alpha_ka_{n-k}, \qquad n \geq k,
\end{equation}
with $a_0=a_1=\cdots=a_{k-2}=0$ and $a_{k-1}=1$, where the $\alpha_i$ are fixed real numbers and $\alpha_k \neq 0$.
The characteristic equation associated with the sequence $a_n$ is defined by
\begin{equation}\label{eq1}
f(x):=x^k-\alpha_1x^{k-1}-\alpha_2x^{k-2}-\cdots-\alpha_{k}=0.
\end{equation}

Let $\lambda_1,\lambda_2,\ldots,\lambda_k$ denote the roots of \eqref{eq1}.  By \cite[Lemma 5.2]{MS}, we have
\begin{equation}\label{eq2}
a_n=c_1\lambda_1^n+c_2\lambda_2^n+\cdots+c_k\lambda_k^n, \qquad n \geq 0,
\end{equation}
where
$$c_i=\frac{1}{\prod_{j=1,j\neq i}^k(\lambda_i-\lambda_j)}, \qquad 1 \leq i \leq k,$$
whenever the $\lambda_i$ are distinct.
Upon writing
$$f(x)=(x-\lambda_1)(x-\lambda_2)\cdots (x-\lambda_k),$$
we have by the product rule of differentiation that
$$f'(\lambda_i)=\prod_{j=1,j\neq i}^k(\lambda_i-\lambda_j)=\frac{1}{c_i}, \qquad 1 \leq i \leq k.$$

\begin{definition}
A zero of a polynomial g will be called dominant if it is simple and is strictly greater in modulus than all of its other zeros.
\end{definition}

Note that if $g$ has real coefficients, then a dominant zero must be real since non-real zeros come in conjugate pairs.

\begin{lemma}\label{l0}
If $f(x)$ defined by \eqref{eq1} has a dominant zero $\lambda$, then $\lambda>0$ and $f'(\lambda)>0$.
\end{lemma}
\begin{proof}
Suppose $\lambda=\lambda_1$.  Define
\begin{equation}\label{l0e1}
e_n=\frac{\sum_{i=2}^kc_i\lambda_i^n}{c_1\lambda_1^n}, \qquad n \geq 0.
\end{equation}
Note that $a_n=c_1\lambda_1^n(1+e_n)$, by \eqref{eq2}.  Thus $\lambda_1$ and $c_1=\frac{1}{f'(\lambda_1)}$ real implies $e_n$ is real.  Note further that $e_n\rightarrow 0$ as $n \rightarrow \infty$ since $\lambda_1$ is dominant.  Taking $n$ to be large and even implies $c_1>0$ and thus $f'(\lambda_1)=\frac{1}{c_1}>0$. Taking $n$ to be large and odd then implies $\lambda_1$ is positive.
\end{proof}

The following limit holds for the numbers $e_n$.

\begin{lemma}\label{l1}
Suppose that the polynomial $f(x)$ defined by \eqref{eq1} has dominant zero $\lambda$.  Then we have
\begin{equation}\label{l1e1}
\lim_{n \rightarrow \infty}(1+e_n)^{p(n)}=1,
\end{equation}
for any polynomial $p(n)$.
\end{lemma}
\begin{proof}
We provide a proof only in the case when the $\lambda_i$ are distinct, the proof in the case when some of the $\lambda_i$ are repeated being similar.
We will show
\begin{equation}\label{l1e2}
\lim_{n \rightarrow \infty}(1+|e_n|)^{p(n)}=\lim_{n \rightarrow \infty}(1-|e_n|)^{p(n)}=1,
\end{equation}
from which \eqref{l1e1} follows.  (Note that $1-|e_n|$ is positive for $n$ sufficiently large, which implies that the expression $(1-|e_n|)^{p(n)}$ is real for all such $n$.)  Let
$$r=\frac{\max\{|\lambda_2|,|\lambda_3|,\ldots,|\lambda_k|\}}{\lambda_1}$$
and
$$M=\frac{\max\{|c_2|,|c_3|,\ldots,|c_k|\}}{c_1}.$$

Note that $$|e_n|\leq (k-1)Mr^{n}, \qquad n \geq 0,$$
so to show \eqref{l1e2}, we only need to show
\begin{equation}\label{l1e3}
\lim_{n \rightarrow \infty}(1+cr^n)^{p(n)}=\lim_{n \rightarrow \infty}(1-cr^n)^{p(n)}=1,
\end{equation}
for constants $c>0$ and $0 < r < 1$.  The limits in \eqref{l1e3} can be evaluated by taking a logarithm and applying l'H\^{o}pital's rule, which completes the proof.
\end{proof}

\begin{theorem}\label{th1} Suppose that the characteristic polynomial $f(x)$ associated with the sequence $a_n$ has dominant zero $\lambda$ such that $f'(\lambda)>1$. Then the sequence of ratios $\frac{\sqrt[n]{a_{n}}}{\sqrt[n-1]{a_{n-1}}}$ is strictly decreasing for all $n \geq N$, for some $N$ depending on the $\alpha_i$, and has limit $1$. \end{theorem}
\begin{proof}
We provide a proof only in the case when the $\lambda_i$ are distinct.
First observe that
$$\frac{\sqrt[n]{a_{n}}}{\sqrt[n-1]{a_{n-1}}}>\frac{\sqrt[n+1]{a_{n+1}}}{\sqrt[n]{a_{n}}}$$
if and only if
$$\left[c_1\lambda_1^n(1+e_n)\right]^{2/n}>\left[c_1\lambda_1^{n+1}(1+e_{n+1})\right]^{1/(n+1)}\left[c_1\lambda_1^{n-1}(1+e_{n-1})\right]^{1/(n-1)},$$
which may be rewritten as
\begin{equation}\label{th1e2}
\frac{(1+e_n)^{2(n^2-1)}}{(1+e_{n-1})^{n(n+1)}(1+e_{n+1})^{n(n-1)}}>c_1^2.
\end{equation}
By Lemma \ref{l1}, we have
$$\lim_{n \rightarrow \infty}(1+e_n)^{2(n^2-1)}=\lim_{n \rightarrow \infty}(1+e_{n-1})^{n(n+1)} =\lim_{n \rightarrow \infty}(1+e_{n+1})^{n(n-1)}=1,$$
which implies \eqref{th1e2} since $c_1=\frac{1}{f'(\lambda_1)}<1$.

For the last statement, note that
$$\log\left(\frac{\sqrt[n+1]{a_{n+1}}}{\sqrt[n]{a_n}}\right)=\frac{1}{n+1}\log a_{n+1}-\frac{1}{n}\log a_n=\frac{\log c_1+\log(1+e_{n+1})}{n+1}-\frac{\log c_1+\log(1+e_{n})}{n},$$
and take limits as $n \rightarrow \infty$.
\end{proof}

\begin{corollary}\label{t1c1}
If $a_n$ is a sequence such that $f(x)$ has a dominant zero $\lambda$ satisfying $f'(\lambda)>1$, then $\sqrt[n]{a_n}$ is strictly increasing for all sufficiently large $n$.
\end{corollary}

\emph{Remark:} If we allow the sequence $a_n$ to contain negative terms, then modifying slightly the proof of Theorem \ref{th1} yields the result for $|a_n|$.

Let us exclude for now from consideration recurrences of the form
$$a_n=\alpha_da_{n-d}+\alpha_{2d}a_{n-2d}+\cdots+\alpha_ka_{n-k}, \qquad n \geq k,$$
for some divisor $d>1$ of $k$ and subject to the same initial conditions.  Observe that such recurrences may be reduced, upon letting $b_m=a_{dm+d-1}$, to those of the form
$$b_m=\alpha_db_{m-1}+\alpha_{2d}b_{m-2}+\cdots + \alpha_k b_{m-\frac{k}{d}}, \qquad m \geq \frac{k}{d},$$
where $b_0=b_1=\cdots=b_{\frac{k}{d}-2}=0$ and $b_{\frac{k}{d}-1}=1$ (note that $a_{dm+r}=0$ for all $m$ if $0 \leq r < d-1$, by the initial conditions).

We now describe a class of recurrences frequently arising in applications for which the characteristic polynomial has a dominant zero.

\begin{lemma}\label{l2}
Suppose that $\alpha_i \geq 0$ for all $i$ in \eqref{eq0} with $\alpha_k\neq0$ and furthermore that it is not the case that $\alpha_i=0$ for all $i \in [k]-\{d,2d,\ldots,k\}$ for some divisor $d>1$ of $k$. Then  $f(x)$ has a dominant zero.
\end{lemma}
\begin{proof}
Let $f(x)=x^{k}-\alpha_1x^{k-1}-\cdots - \alpha_k$, where the $\alpha_i$ satisfy the given hypotheses.  By Descartes' rule of signs, the equation $f(x)=0$ has a single (simple) positive root, which we will denote by $\lambda$.  Let $\rho$ be any root of the equation $f(x)=0$ other than $\lambda$.  We will show that the numbers $\alpha_i \rho^{k-i}$, $1 \leq i \leq k$, cannot all be non-negative real numbers.  Suppose, to the contrary, that this is the case.  Let $\{i_1,i_2,\ldots,i_a\}$ denote the set of indices $i$ such that $\alpha_i \neq 0$.  Let $b=\min\{i_{j+1}-i_j: 1 \leq j \leq a-1\}$ and $\ell$ be an index such that $i_{\ell+1}-i_{\ell}=b$.  Then $$\alpha_{i_{\ell+1}}\rho^{i_{\ell+1}}=r\alpha_{i_\ell}\rho^{i_\ell}$$
for some $r>0$ implies $$\rho=\left(\frac{r\alpha_{i_\ell}}{\alpha_{i_{\ell+1}}}\right)^{1/b}\xi,$$
where $\xi$ denotes a primitive $b'$-th root of unity for some positive divisor $b'$ of $b$.  Note that $b'>1$ since $f(x)$ has only one positive real zero.  If $b'$ does not divide $k$, then $\rho^k$ is not a positive real since $\xi^k\neq 1$ in this case.
But this contradicts the equality $\rho^k=\alpha_1\rho^{k-1}+\cdots+\alpha_k$, since the right-hand side is a positive real.  Thus  $b'$ divides $k$ and so it must be the case that there exists some index $m$ such that the difference $c=i_{m+1}-i_m$ is not divisible by $b'$ (for otherwise, the second hypothesis concerning the $\alpha_i$ would be contradicted).  But then
$$\alpha_{i_{m+1}}\rho^{\alpha_{i_{m+1}}}=s\alpha_{i_m}\rho^{\alpha_{i_m}}$$
for some $s>0$ implies $\rho^c$ is a positive real number and hence $\xi^c=1$, which implies $b'$ divides $c$, a contradiction.

Thus, the $\alpha_i\rho^{k-i}$ cannot all be non-negative real numbers.  Suppose $i'$ is such that $\alpha_{i'}\rho^{k-i'}$ is either negative or not real.  Note that the assumption $\alpha_k>0$ implies $i'<k$.  Then we may write
\begin{align*}
|\rho|^k&=|\rho^k|=\left|\sum_{i=1}^k \alpha_i\rho^{k-i}\right|=\left|\alpha_k+\alpha_{i'}\rho^{k-i'}+\sum_{i=1,i\neq i'}^{k-1} \alpha_i\rho^{k-i}\right|\\
&\leq\left|\alpha_k+\alpha_{i'}\rho^{k-i'}\right|+\left|\sum_{i=1,i\neq i'}^{k-1} \alpha_i\rho^{k-i}\right|\leq\left|\alpha_k+\alpha_{i'}\rho^{k-i'}\right|+\sum_{i=1,i\neq i'}^{k-1}\alpha_i|\rho|^{k-i}\\
&<\alpha_k+\alpha_{i'}|\rho|^{k-i'}+\sum_{i=1,i\neq i'}^{k-1}\alpha_i|\rho|^{k-i}=\sum_{i=1}^k\alpha_i|\rho|^{k-i},
\end{align*}
where the last inequality is strict since $\alpha_{i'}\rho^{k-i'}$ is not a positive real number.  But then we have $|\rho|^k<\sum_{i=1}^k\alpha_i|\rho|^{k-i}$, which implies $f(|\rho|)<0$.  Since $f(x)>0$ if $x>\lambda$ and $f(x)<0$ if $0<x<\lambda$, it follows that $|\rho|<\lambda$, as desired.
\end{proof}

\emph{Remark:}  By Theorem \ref{th1}, for sequences $a_n$ defined by a recurrence of the form \eqref{eq0}, where the $\alpha_i$ satisfy the hypotheses of Lemma \ref{l2}, one needs only to verify the condition $f'(\lambda)>1$ in order to establish the log-concavity of $\sqrt[n]{a_n}$ for large $n$.

We now apply the previous results to the sequences $\sqrt[n]{f_{n}^{(k)}}$ and $\sqrt[n]{g_{n}^{(k)}}$ where $k \geq 2$.

\begin{theorem}\label{th2}
The characteristic polynomial $f(x)$ associated with either the sequence $f_n^{(k)}$ or $g_n^{(k)}$ has a dominant zero $\lambda$ such that $f'(\lambda)>1$.  Thus, for $k \geq 2$, the sequences $\sqrt[n]{f_{n}^{(k)}}$ and $\sqrt[n]{g_n^{(k)}}$ are log-concave for all $n\geq N$ for some constant $N$ depending on $k$.
\end{theorem}
\begin{proof}
We need only to verify the first statement in each case.  Note that both $f_n^{(k)}$ and $g_n^{(k)}$ are defined by recurrences such that the constants $\alpha_i$ satisfy the conditions given in Lemma \ref{l2}. Thus, we need only to verify $f'(\lambda)>1$.  In the case of $f_n^{(k)}$, this follows easily since
\begin{align*}
f'(\lambda)&=k\lambda^{k-1}-(k-1)\lambda^{k-2}-\cdots-1=k\left(\lambda^{k-2}+\lambda^{k-3}+\cdots+\frac{1}{\lambda}\right)-(k-1)\lambda^{k-2}-\cdots-1\\
&=\frac{k}{\lambda}+\lambda^{k-2}+2\lambda^{k-3}+\cdots+(k-1)>1.
\end{align*}
In the case of $g_n^{(k)}$, note that $\lambda>1$ since $f(1)<0$.  Then $f'(\lambda)=\lambda^{k-2}(1+k(\lambda-1))>1$ since $\lambda>1$, which completes the proof.
\end{proof}

\section{Third and fourth order sequences}

In this section, we will determine the smallest possible $N$ in Theorem \ref{th1} in some particular cases.  The method illustrated here can be applied to other sequences in finding the smallest $N$.  Let us denote the $k=3$ cases of the sequences $f_n^{(k)}$ and $g_n^{(k)}$ by $t_n$ and $r_n$, respectively.  The $t_n$ and $r_n$ are known as the \emph{tribonacci} and $3$-\emph{bonacci} numbers, respectively.  See, e.g., \cite[Section 3.3]{BQ} and also the sequences A000073 and A000930 in \cite{Sl}.

We have the following estimates for the values of the $c_i$ and $\lambda_i$ in \eqref{eq2} in the cases of $t_n$ and $r_n$.

Values corresponding to the sequence $t_n$:
\begin{align*}
c_1&=0.182803,\quad
c_2=-0.091401+0.340546i \quad \text{and} \quad c_3=\overline{c_2},\\
\lambda_1&=1.839286, \quad
\lambda_2=-0.419643+0.606290i \quad \text{and} \quad
\lambda_3=\overline{\lambda_2}.\\
\end{align*}

Values corresponding to the sequence $r_n$:
\begin{align*}
c_1&=0.284693,\quad
c_2=-0.142346+0.305033i \quad \text{and} \quad c_3=\overline{c_2},\\
\lambda_1&=1.465571, \quad \lambda_2=-0.232785+0.792551i \quad \text{and} \quad
\lambda_3=\overline{\lambda_2}.\\
\end{align*}

We will make use of these estimates in the proof of the following result.

\begin{theorem}\label{th3}
The ratio $\frac{\sqrt[n]{a_{n}}}{\sqrt[n-1]{a_{n-1}}}$ is strictly decreasing for all $n \geq 4$ when $a_n=t_n$ and for all $n\geq 8$ when $a_n=r_n$.
\end{theorem}
\begin{proof}
We first consider the case $t_n$.  One can verify by direct computation that
$$\frac{\sqrt[n]{t_{n}}}{\sqrt[n-1]{t_{n-1}}}>\frac{\sqrt[n+1]{t_{n+1}}}{\sqrt[n]{t_{n}}}$$
for $4 \leq n \leq 9$, so we may assume $n \geq 10$.  By \eqref{th1e2}, it suffices to show
\begin{equation}\label{th3e1}
(1+e_n)^{2(n^2-1)}>c_1^{2/3}, \quad (1+e_{n-1})^{n(n+1)}<c_1^{-2/3} \quad \text{and} \quad (1+e_{n+1})^{n(n-1)}<c_1^{-2/3},
\end{equation}
for $n \geq 10$.

To do so, first note that
\begin{align*}
|e_n|&=\left|\frac{2\text{Re}(c_2\lambda_2^n)}{c_1\lambda_1^n}\right|\leq \frac{2|c_2|}{c_1}\left(\frac{|\lambda_2|}{\lambda_1}\right)^n=\frac{|\lambda_1-\lambda_2|}{|\text{Im}(\lambda_2)|}\left(\frac{|\lambda_2|}{\lambda_1}\right)^n< (3.86)(0.41)^{n}.
\end{align*}
Thus, to show \eqref{th3e1}, it is enough to show
\begin{equation}\label{th3e2}
(1-M_n)^{2(n^2-1)}>c_1^{2/3}, \quad (1+M_{n-1})^{n(n+1)}<c_1^{-2/3} \quad \text{and} \quad (1+M_{n+1})^{n(n-1)}<c_1^{-2/3},
\end{equation}
where $M_n=(3.86)(0.41)^{n}$.  Since $M_{n}$ is a decreasing positive sequence, we have $(1+M_{n-1})^{n(n+1)}>(1+M_{n+1})^{n(n-1)}$, so we only need to show the first two inequalities in \eqref{th3e2}.

The first inequality in \eqref{th3e2} holds if and only if $\log (1-M_n)>\frac{\log c_1}{3(n^2-1)}$. For this last inequality, we can show
\begin{equation}\label{th3e3}
M_n+M_n^2<-\frac{\log (0.19)}{3(n^2-1)}, \qquad n \geq 10,
\end{equation}
since $c_1<0.19$ and $-\log(1-y)<y+y^2$ for $0<y<\frac{1}{2}$.  To show \eqref{th3e3}, let $a(x)=-\frac{\log (0.19)}{3(x^2-1)}$ and $b(x)= M_x+M_x^2$, where $M_x$ has the obvious meaning.  Observe that $a(10)>b(10)$ and $\lim_{x \rightarrow \infty}(a(x)-b(x))=0$.  Thus to prove $a(x)>b(x)$ for $x \geq 10$, it suffices to show $a'(x)<b'(x)$ for $x \geq 10$.  Since $\frac{2}{3x^3}<\frac{2x}{3(x^2-1)^2}$, it is enough to show
$$\frac{(3.86)\log(0.41)}{\log(0.19)}(0.41)^{x}+\frac{2(3.86)^2\log(0.41)}{\log(0.19)}(0.41)^{2x}<\frac{2}{3x^3},$$
and for this, it is enough to show
\begin{equation}\label{th3e4}
\frac{\log(0.19)}{(3.86)\log(0.41)}(0.41)^{-x}>3x^3, \qquad x \geq 10.
\end{equation}
Note that \eqref{th3e4} holds for $x=10$, with the derivative of the difference of the two sides seen to be positive for all $x \geq 10$.  This finishes the proof of the first inequality in \eqref{th3e2}.

We proceed in a similar manner to verify the second inequality in \eqref{th3e2}. Since $\log(1+y)<y$ for $y>0$, it suffices to show $c(x)>d(x)$ for $x \geq 10$, where $c(x)=-\frac{2\log(0.19)}{3x(x+1)}$ and $d(x)=M_{x-1}$.  Since $c(10)>d(10)$ and $\lim_{x \rightarrow \infty}(c(x)-d(x))=0$, we only need to show that $c'(x)<d'(x)$ for $x \geq 10$.  Now $c'(x)<d'(x)$ if and only if
\begin{equation}\label{th3e5}
\frac{2(2x+1)}{3x^2(x+1)^2}>\frac{(3.86)\log(0.41)}{\log(0.19)}(0.41)^{x-1}, \qquad x \geq 10.
\end{equation}
Since
$$\frac{2(2x+1)}{3x^2(x+1)^2}>\frac{2(2x+1)}{3\left(x+\frac{1}{2}\right)^4}=\frac{4}{3\left(x+\frac{1}{2}\right)^3},$$
to prove \eqref{th3e5}, one can show
$$(0.41)^{1-x}>\frac{3}{4}(2.08)\left(x+\frac{1}{2}\right)^3, \qquad x \geq 10,$$
which can be done by comparing the derivatives of the two sides.  This establishes the second inequality in \eqref{th3e2} and completes the proof in the case when $a_n=t_n$.

A similar proof can be given when $a_n=r_n$, which we outline as follows.  We first verify by computation that
$$\frac{\sqrt[n]{r_{n}}}{\sqrt[n-1]{r_{n-1}}}>\frac{\sqrt[n+1]{r_{n+1}}}{\sqrt[n]{r_{n}}}$$
for $8 \leq n \leq 17$.  Thus, we may assume $n \geq 18$ in showing \eqref{th3e1} for $r_n$.  We use the bounding function of $M_n=(2.37)(0.57)^n$ in proving the first two inequalities in \eqref{th3e2}.   For the first inequality, instead of \eqref{th3e4}, one needs to show
$$\frac{\log(0.29)}{(2.37)\log(0.57)}(0.57)^{-x}>3x^3, \qquad x \geq 18,$$
which can be done by a comparison of the derivatives of the two sides.  In proving the second inequality in \eqref{th3e2} above for $r_n$, it is enough to verify
$$(0.57)^{1-x}>\frac{3}{4}(1.08)\left(x+\frac{1}{2}\right)^3, \qquad x \geq 18.$$
This can be done by comparing derivatives of the two sides for $x \geq 18$, which completes the proof in the $r_n$ case.
\end{proof}

By Theorems \ref{th1} and \ref{th3} and direct computation, we obtain the following.

\begin{corollary}\label{th3c1}
The sequence $\sqrt[n]{a_n}$ is strictly increasing for $n\geq 5$ when $a_n=t_n$ or $r_n$.
\end{corollary}

Let $p_n$ and $q_n$ denote the respective $k=4$ cases of the $f_n^{(k)}$ and $g_n^{(k)}$.  The $p_n$ and $q_n$ are known as the \emph{tetranacci} and $4$-\emph{bonacci} numbers and occur, respectively, as sequences A000078 and A017898 in \cite{Sl}.  A proof comparable to the previous one yields the following result.

\begin{theorem}\label{th4}
The ratio $\frac{\sqrt[n]{a_{n}}}{\sqrt[n-1]{a_{n-1}}}$ is strictly decreasing for all $n \geq 5$ when $a_n=p_n$ and for all $n\geq 11$ when $a_n=q_n$.
\end{theorem}

Given the prior two results, one might wonder if one can find some bound for the best possible $N$ as a function of $k$.  In the case of $f_n^{(k)}$, such a bound seems possible in light of the fact (see \cite[Lemma 5.2]{MS}) that the dominant zero of the associated characteristic polynomial approaches $2$ as $k$ approaches infinity, with all other zeros of modulus strictly less than $1$ and distinct.  By the present method, one would need an estimate of the magnitude of the constants $c_i$ in \eqref{eq2}.  In particular, it would be useful to have a lower bound (as a function of $k$) for the quantity $$m(k):=\min_{2\leq i\leq k}\left|\prod_{j=1,j\neq i}^k(\lambda_i-\lambda_j)\right|.$$
If $m(k)$ can be shown, for example, to be no smaller than $ab^{-k}$ for some constants $a$ and $b$ with $b>\frac{1}{2}$, then a bound for $N$ in terms of $k$ could probably be obtained.


\begin{thebibliography}{20}

\bibitem{BQ}
A. T. Benjamin and J. J. Quinn, \emph{Proofs that Really Count: The Art of Combinatorial Proof}, Mathematical Association of America, 2003.

\bibitem{Br}
F. Brenti, Log-concave and unimodal sequences in algebra, combinatorics, and geometry: an update, \emph{Contemp. Math.} \textbf{178} (1994) 71--89.

\bibitem{CGW}
W. Y. C. Chen, J. J. F. Guo and L. X. W. Wang, Zeta functions and the log-behavior of combinatorial sequences, \emph{Proc. Edinb. Math. Soc.} (2), in press.

\bibitem{HSW}
Q.-H. Hou, Z.-W. Sun and H. Wen, On monotonicity of some combinatorial sequences, \emph{Publ. Math. Debrecen}, in press, arXiv:1208.3903.

\bibitem{Kn}
D. E. Knuth, \emph{The Art of Computer Programming: Sorting and Searching}, Vol. 3, Addison-Wesley, 1973.

\bibitem{LW}
L. L. Liu and Y. Wang, On the log-convexity of combinatorial sequences, \emph{Adv. in Appl. Math.} \textbf{39} (2007) 453--476.

\bibitem{MS}
T. Mansour and M. Shattuck, Polynomials whose coefficients are $k$-Fibonacci numbers, \emph{Ann. Math. Inform.} \textbf{40} (2012) 57--76.

\bibitem{Mu}
E. Munarini, A combinatorial interpretation of the generalized Fibonacci numbers, \emph{Adv. in Appl. Math.} \textbf{19} (1998) 306--318.

\bibitem{Sl}
N. J. A. Sloane, The On-Line Enyclopedia of Integer Sequences, available at http://oeis.org, 2010.

\bibitem{St}
R. P. Stanley, Log-concave and unimodal sequences in algebra, combinatorics, and geometry, \emph{Ann. New York Acad. Sci.} \textbf{576} (1989) 500--534.

\bibitem{Sun0}
Z.-W. Sun, On a sequence involving sums of primes, \emph{Bull. Aust. Math. Soc.} \textbf{88} (2013) 197--205.

\bibitem{Sun}
Z.-W. Sun, Conjectures involving arithmetical sequences, \emph{Number Theory: Arithmetic in Shangri-La}, Proceedings of the 6th China-Japan Seminar (Shaghai, 2011), \emph{World Scientific} (2013) 244-258.

\bibitem{WZ}
Y. Wang and B. X. Zhu, Proofs of some conjectures on monotonicity of number theoretic and combinatorial sequences, \emph{Sci. China Math.} \textbf{57} (2014) 2429--2435.


\end{thebibliography}
\end{document}